\documentclass[preprint,11pt]{elsarticle}
\usepackage{amssymb,amsthm,amsmath}
\usepackage{enumerate}
\usepackage{epsfig,graphicx}
\usepackage{algorithm}
\usepackage{algpseudocode}
\usepackage{algorithmicx}

\journal{Journal of Computational and Applied Mathematics}

\newtheorem{corollary}{Corollary}
\newtheorem{example}{Example}
\newtheorem{lemm}{Lemma}
\newtheorem{thm}{Theorem}
\newtheorem{prp}{Proposition}

\begin{document}

\begin{frontmatter}

\title{Computation of isotopisms of algebras over finite fields by means of graph invariants}
\author{O. J. Falc\'on$^{1}$}
\ead{oscfalgan@yahoo.es}
\author{R. M. Falc\'on$^2$}
\ead{rafalgan@us.es}
\author{J. N\'u\~nez$^1$}
\ead{jnvaldes@us.es}
\author{A. M. Pacheco$^{3}$}
\ead{ampacheco@uloyola.es}
\author{M. T. Villar$^1$}
\ead{villar@us.es}
\address{$^1$ Department of Geometry and Topology. University of Seville, Spain.\\
$^2$ Department of Applied Mathematics I. University of Seville, Spain.\\
$^3$ Department of Quantitative Methods. Loyola University Andalusia, Spain.
}

\begin{abstract} In this paper we define a pair of faithful functors that map isomorphic and isotopic finite-dimensional algebras over finite fields to isomorphic graphs. These functors reduce the cost of computation that is usually required to determine whether two algebras are isomorphic. In order to illustrate their efficiency, we determine explicitly the classification of two- and three-dimensional partial quasigroup rings.
\end{abstract}

\end{frontmatter}

\section{Introduction}

Graph invariants constitute an interesting tool in Chemistry, Communication or Engineering \cite{Dobrynin2001, Khalifeh2010, Yousefi2011}. In Mathematics, one of the topics for which graph invariants have revealed to play an important role is the classical problem of deciding whether two algebras are isomorphic. This problem is usually dealt with by computing the reduced Gr\"obner basis of the system of polynomial equations that is uniquely related to the structure constants of both algebras. This computation is, however, very sensitive to the number of variables \cite{Gao2009} and gives rise to distinct problems of computation time and memory usage even for low-dimensional algebras \cite{Falcon2016a, Graaf2005}. This paper deals with Graph Theory in order to reduce this cost of computation.

Graph invariants have been proposed in the last years as an efficient alternative to study isomorphisms of distinct types of algebras \cite{Bocian2014, Ceballos2016, Kaveh2011}. Nevertheless, the problem of identifying a functor that relates the category of algebras with that of graphs remains still open. Based on a proposal of McKay et al. \cite{McKay2007} for identifying isotopisms of Latin squares with isomorphisms of vertex-colored graphs, we describe in Section 3 a pair of graphs that enable us to find faithful functors between finite-dimensional algebras over finite fields and these types of graphs. These functors map isomorphic and isotopic algebras to isomorphic graphs. Reciprocally, any pair of isomorphic graphs is uniquely related to a pair of algebras so that there exists a multiplicative map between them. The main advantage of our proposal, apart from the reduction of the mentioned cost of computation, is the feasibility of studying the possible isomorphism between two given finite-dimensional algebras defined over the same field, whatever the types of both algebras are. As an illustrative example, we focus in Section 4 on the classification of partial quasigroup rings according to the known isotopism classes of partial Latin squares on which they are based.

\section{Preliminaries}

In this section we expose some basic concepts and results on Graph Theory, isotopisms of algebras, partial Latin squares and Computational Algebraic Geometry that we use throughout the paper. For more details about these topics we refer, respectively, to the manuscripts \cite{Harary1969, Albert1942, Denes1974, Cox1998}.

\subsection{Graph Theory}

A {\em graph} is a pair $G=(V,E)$ formed by a set $V$ of {\em vertices} and a set $E$ of $2$-subsets of $V$ called {\em edges}. Two vertices defining an edge are said to be {\em adjacent}. The {\em degree} of a vertex $v$ is the number $d(v)$ of edges containing $v$. The graph $G$ is {\em vertex-colored} if there exists a partition of $V$ into color sets. The color of a vertex $v$ is denoted as $\mathrm{color}(v)$. An {\em isomorphism} between two vertex-colored graphs $G$ and $G'$ is any bijective map $f$ between their sets of vertices that preserves collinearity and color sets, that is, such that it maps edges to edges and $\mathrm{color}(f(v))=\mathrm{color}(v)$, for all vertex $v$ in $G$.

\subsection{Isotopisms of algebras}

Two algebras $A$ and $A'$ over a field $\mathbb{K}$ are said to be {\em isotopic} if there exist three non-singular linear transformations $f$, $g$ and $h$ from $A$ to $A'$ such that $f(u)g(v)=h(uv)$, for all $u,v\in A$. The triple $(f,g,h)$ is an {\em isotopism} between $A$ and $A'$. If $f=g=h$, then this constitutes an {\em isomorphism}.

\vspace{0.1cm}

The {\em structure constants} of an $n$-dimensional algebra $A$ over a field $\mathbb{K}$ of basis $\{e_1,\ldots,e_n\}$ are the numbers $c_{ij}^k\in \mathbb{K}$ such that $e_ie_j = \sum_{k=1}^n c_{ij}^k e_k$, for all $i, j \leq n$. If all of them are zeros, then $A$ is {\em abelian}. In particular, the $n$-dimensional abelian algebra is not isotopic to any other $n$-dimensional algebra.

\vspace{0.2cm}

The {\em left annihilator} of a vector subspace $S$ of the algebra $A$ is the set $\mathrm{Ann}_{A^-}(S)=\{u\in A\mid\, uv=0, \text { for all } v\in S\}$. Its {\em right annihilator} is the set $\mathrm{Ann}_{A^+}(S)=\{u\in A\mid\, vu=0, \text { for all } v\in S\}$. The intersection of both sets is the {\em annihilator} $\mathrm{Ann}_A(S)$.

\begin{lemm} \label{lemm_annihilator} Let $(f,g,h)$ be an isotopism between two $n$-dimensional algebras $A$ and $A'$, and let $S$ be a vector subspace of $A$. Then,
\begin{enumerate}[a)]
\item $f(\mathrm{Ann}_{A^-}(S)) = \mathrm{Ann}_{{A'}^-}(g(S))$.
\item $g(\mathrm{Ann}_{A^+}(S)) = \mathrm{Ann}_{{A'}^+}(f(S))$.
\item $f(\mathrm{Ann}_{A^-}(S))\cap g(\mathrm{Ann}_{A^+}(S)) = \mathrm{Ann}_{A'}(f(S)\cap g(S)).$
\end{enumerate}
\end{lemm}

\begin{proof} Let us prove assertion (a). Assertion (b) follows similarly and assertion (c) is a consequence of (a) and (b). Let $u\in g(S)$ and $v\in f(\mathrm{Ann}_{A^-}(S))$. Then, $vu=f(f^{-1}(v))g(g^{-1}(u))=h(f^{-1}(v)g^{-1}(u))=h(0)=0$, because $g^{-1}(u)\in S$ and $f^{-1}(v)\in \mathrm{Ann}_{A^-}(S)$. Hence, $f(\mathrm{Ann}_{A^-}(S))$ $\subseteq \mathrm{Ann}_{{A'}^-}(g(S))$. Now, let $u\in \mathrm{Ann}_{{A'}^-}(g(S))$ and $v\in S$. From the regularity of $f$, we have that $h(f^{-1}(u)v)=ug(v)=0$. The regularity of $h$ involves that $f^{-1}(u)v=0$. Thus, $u\in f(\mathrm{Ann}_{A^-}(S))$ and hence, $\mathrm{Ann}_{{A'}^-}(g(S))\subseteq f(\mathrm{Ann}_{A^-}(S))$.
\end{proof}

\vspace{0.2cm}

The {\em derived algebra} of $A$ is the subalgebra
$A^2=\{uv\mid\, u,v\in A\}\subseteq A$.

\begin{lemm} \label{lemm_M1} Let $(f,g,h)$ be an isotopism between two $n$-dimensional algebras $A$ and $A'$. Then, $h(A^2)=A'^2$.
\end{lemm}

\begin{proof}  The regularity of $f$ and $g$ involves that $f(A)=g(A)=A'$ and hence, $A'^2=f(A)g(A)=h(A^2)$.
\end{proof}

\vspace{0.2cm}

Let $\cdot$ be a partial binary operation over the set $[n]=\{1,\ldots,n\}$. The pair $([n],\cdot)$ is called a {\em partial magma} of {\em order} $n$. It is {\em isotopic} to a partial magma $([n],\circ)$ if there exist three permutations $\alpha$, $\beta$ and $\gamma$ in the symmetric group $S_n$ such that $\alpha(i)\circ\beta(j)=\gamma(i\cdot j)$, for all $i,j\leq n$ such that $i\cdot j$ exists. If $\alpha=\beta=\gamma$, then the partial magmas are said to be {\em isomorphic}. The triple $(\alpha,\beta,\gamma)$ is an {\em isotopism} of partial magmas (an {\em isomorphism} if $\alpha=\beta=\gamma$).

\vspace{0.2cm}

A {\em partial magma algebra} $A^{\cdot}$ {\em based on} a partial magma $([n],\cdot)$ is an $n$-dimensional algebra over a field $\mathbb{K}$ such that there exists a basis $\{e_1,\ldots,e_n\}$ satisfying that, if $i\cdot j$ exists for some pair of elements $i,j\leq n$, then $e_ie_j=c_{ij}e_{i\cdot j}$ for some non-zero structure constant $c_{ij}\in \mathbb{K}\setminus\{0\}$. If all the structure constants are equal to $1$, then this is called a {\em partial magma ring}.

\begin{lemm}\label{lemm_partial_magma} Two partial magma rings are isotopic (isomorphic, respectively) if their respective partial magmas on which they are based are isotopic (isomorphic, respectively).
\end{lemm}

\begin{proof} Let $A^{\cdot}$ and $A^{\circ}$ be two partial magma rings based, respectively, on two isotopic partial magmas $([n],\cdot)$ and $([n],\circ)$. Let $\{e_1,\ldots,e_n\}$ and $\{e'_1,\ldots,e'_n\}$ be the respective bases of these two algebras and let $(f,g,h)$ be an isotopism between their corresponding partial magmas. For each $\alpha\in\{f,g,h\}$, let us define the map $\overline{\alpha}(e_i)=e'_{\alpha(i)}$. Then, $\overline{f}(e_i)\overline{g}(e_j)=e'_{f(i)}e'_{g(j)}= e'_{f(i)\circ g(j)}=e'_{h(i\cdot j)} =\overline{h}(e_{i\cdot j})= \overline{h}(e_ie_j)$. From linearity, the triple $(\overline{f},\overline{g},\overline{h})$ determines an isotopism between $A^{\cdot}$ and $A^{\circ}$. If $f=g=h$, then this constitutes an isomorphism.
\end{proof}

\vspace{0.2cm}

The reciprocal of Lemma \ref{lemm_partial_magma} is not true in general. Thus, for instance, the two partial magmas $([2],\cdot)$ and $([2],\circ)$ that are respectively described by the non-zero products $1\cdot 1=1$ and $1\circ 1 = 1 = 2\circ 1$ are not isotopic. Nevertheless, the partial magma rings $A^{\cdot}$ and $A^{\circ}$, with respective bases $\{e_1,e_2\}$ and $\{e'_1,e'_2\}$, are isotopic by means of the isotopism $(f,\mathrm{Id},\mathrm{Id})$, where the linear transformation $f$ is described by $f(e_1)=e'_1$ and $f(e_2) = e'_2-e'_1$.

\subsection{Partial Latin squares}

A {\em partial quasigroup} is a partial magma $([n],\cdot)$ such that if the equations $ix=j$ and $yi=j$, with $i,j\in [n]$, have solutions for $x$ and $y$ in $[n]$, then these solutions are unique. The concepts of {\em partial quasigroup algebras} and {\em partial quasigroup rings} arise similarly to those of partial magma algebras and rings. Lemma \ref{lemm_partial_magma} also holds analogously for partial quasigroup rings. Every partial quasigroup of order $n$ constitutes the multiplication table of a {\em partial Latin square} of order $n$, that is, an $n \times n$ array in which each cell is either empty or contains one element chosen from the set $[n]$, such that each symbol occurs at most once in each row and in each column. Every isotopism of a partial quasigroup is uniquely related to a permutation of the rows, columns and symbols of the corresponding partial Latin square. The distribution of partial Latin squares into isotopism classes is known for order up six \cite{Falcon2013, Falcon2015a}. In this paper we make use of graph invariants to study which ones of the known non-isotopic classes of partial Latin squares of order $n\leq 3$ give rise to isotopic classes of partial quasigroup rings over the finite fields $\mathbb{F}_2$ and $\mathbb{F}_3$. In this regard, it is straightforwardly verified that there exists only two one-dimensional partial quasigroup rings: the abelian and that one described by the product $e_1e_1=e_1$. They constitute distinct isotopism classes.

\vspace{0.2cm}

Let $L=(l_{ij})$ be a partial Latin square of order $n$ without empty cells (that is, a {\em Latin square}). McKay et al. \cite{McKay2007} defined the vertex-colored graph $G(L)$ with $n^2+3n$ vertices $\{r_i\mid\, i\leq n\}\cup\{c_i\mid\, i\leq n\}\cup\{s_i\mid\, i\leq n\}\cup \{t_{ij}\mid\, i,j\leq n\}$, where each of the four subsets (related to the rows ($r_i$), columns ($c_i$), symbols ($s_i$) and cells ($t_{ij}$) of the Latin square $L$) has a different color, and $3n^2$ edges $\{r_it_{ij},c_jt_{ij},s_{l_{ij}}t_{ij}\mid\, i,j\leq n\}\}$ (see Figure \ref{Fig_LS}, where we have used distinct styles ($\circ$, $\blacktriangle$, $\blacktriangleright$, $\blacktriangleleft$ and $\bullet$) to represent the colors of the vertices). Two Latin squares $L_1$ and $L_2$ of the same order are isotopic if and only if the graphs $G(L_1)$ and $G(L_2)$ are isomorphic (see Theorem 6 in \cite{McKay2007}).

\begin{figure}[h]
\begin{center}
$\begin{array}{ccc}
\begin{array}{|c|c|}\hline
1 & 2\\ \hline
2 & 1\\ \hline
\end{array} & \equiv & \begin{array}{c} \includegraphics[width=3.5cm]{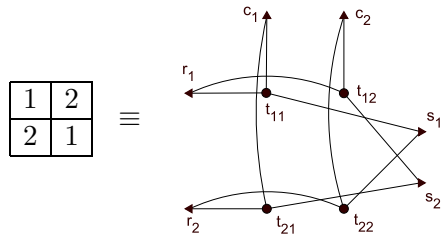}\end{array}\end{array}$
\end{center}
\caption{Graph related to a Latin square of order $2$.}
\label{Fig_LS}
\end{figure}

\subsection{Computational Algebraic Geometry}

Let $\mathbb{K}[X]$ be a multivariate polynomial ring over a field $\mathbb{K}$. The {\em algebraic set} defined by an ideal $I$ of $\mathbb{K}[X]$ is the set $\mathcal{V}(I)$ of common zeros of all the polynomials in $I$. If this set is finite, then the ideal $I$ is {\em zero-dimensional}. This is {\em radical} if every polynomial $f\in \mathbb{K}[X]$ belongs to $I$ whenever there exists a natural number $m$ such that $f^m\in I$. The largest monomial of a polynomial in $I$ with respect to a given monomial term ordering is its {\em leading monomial}. The ideal generated by all the leading monomials of $I$ is its {\em initial ideal}. A {\em standard monomial} of $I$ is any monomial that is not contained in its initial ideal. Regardless of the monomial term ordering, if the ideal $I$ is zero-dimensional and radical, then the number of standard monomials in $I$ coincides with the Krull dimension of the quotient ring $\mathbb{K}[X]/I$ and with the number of points of the algebraic set $\mathcal{V}(I)$. This is computed from the reduced Gr\"obner basis of the ideal. Specifically, a {\em Gr\"obner basis} of the ideal $I$ is any subset $G$ of polynomials in $I$ whose leading monomials generate its initial ideal. This is {\em reduced} if all its polynomials are monic and no monomial of a polynomial in $G$ is generated by the leading monomials of the rest of polynomials in the basis. There exists only one reduced Gr\"obner basis, which can always be computed from Buchberger's algorithm \cite{Buchberger2006}. The computation that is required to this end is extremely sensitive to the number of variables.

\begin{thm}[\cite{Gao2009}, Proposition 4.1.1]\label{Gao} Let $\mathbb{F}_q$ be a finite field, with $q$ a power prime. The complexity time that Buchberger's algorithm requires to compute the reduced Gr\"obner bases of an ideal $\langle\, p_1,\ldots, p_m,p_1^q-p_1,\ldots,p^q_m-p_m\,\rangle$ defined over a polynomial ring $\mathbb{F}_q[x_1,\ldots,x_n]$, where $p_1,\ldots,p_m$ are polynomials given in sparse form and have longest length $l$, is $q^{O(n)}+O(m^2l)$. Here, sparsity refers to the number of monomials.
\end{thm}

\vspace{0.2cm}

Gr\"obner bases can be used to determine the isomorphisms and isotopisms between two $n$-dimensional algebras $A$ and $A'$ over a finite field $\mathbb{F}_q$, with $q$ a prime power, respective basis $\{e_1,\ldots,e_n\}$ and $\{e'_1,\ldots,e'_n\}$, and respective structure constants $c_{ij}^k$ and ${c'}_{ij}^k$. To this end, let us define the sets of variables $\mathfrak{F}_n=\{\mathfrak{f}_{ij}\mid\, i,j\leq n\}$, $\mathfrak{G}_n=\{\mathfrak{g}_{ij}\mid\, i,j\leq n\}$ and $\mathfrak{H}_n=\{\mathfrak{h}_{ij}\mid\, i,j\leq n\}$. These variables play the respective role of the entries in the regular matrices related to a possible isotopism $(f,g,h)$ between the algebras $A$ and $A'$. Here, $\alpha(e_i)=\sum_{j=1}^n \alpha_{ij}e'_j$, for each $\alpha\in\{f,g,h\}$. From the coefficients of each basis vector $e_m$ in the expression $f(e_i)g(e_j)=h(e_ie_j)$, we have that
$$\sum_{k,l=1}^n \mathfrak{f}_{ik}\mathfrak{g}_{jl}{c'}_{kl}^m = \sum_{s=1}^n c_{ij}^s\mathfrak{h}_{sm}, \text{ for all } i,j,m\leq n.$$

\begin{thm}\label{thm_CAG_Isom} The next two assertions hold.
\begin{enumerate}[a)]
\item The isotopism group between the algebras $A$ and $A'$ is identified with the algebraic set of the ideal $I^{\mathrm{Isot}}_{A,A'}$ of $\mathbb{F}_q[\mathfrak{F}_n\cup\mathfrak{G}_n\cup\mathfrak{H}_n]$, which is defined as
    {\small
$$\langle\, \sum_{k,l=1}^n \mathfrak{f}_{ik}\mathfrak{g}_{jl}{c'}_{kl}^m - \sum_{s=1}^n c_{ij}^s\mathfrak{h}_{sm}\mid\, i,j,m\leq n\,\rangle + \langle\,\det(M)^{q-1}-1\mid\, M\in\{F,G,H\}\,\rangle,$$}

\noindent where $F$, $G$ and $H$ denote, respectively, the matrices of entries in  $\mathfrak{F}_n$, $\mathfrak{G}_n$ and $\mathfrak{H}_n$. Besides, $|\mathcal{V}(I^{\mathrm{Isot}}_{A,A'})|= \mathrm{dim}_{\mathbb{F}_q} (\mathbb{F}_q[\mathfrak{F}_n\cup\mathfrak{G}_n\cup\mathfrak{H}_n]/ I^{\mathrm{Isot}}_{A,A'})$.
\item The isomorphism group between the algebras $A$ and $A'$ is identified with the algebraic set of the ideal $I^{\mathrm{Isom}}_{A,A'}$ of $\mathbb{F}_q[\mathfrak{F}_n]$, which is defined as
$$\langle\, \sum_{k,l=1}^n \mathfrak{f}_{ik}\mathfrak{f}_{jl}{c'}_{kl}^m - \sum_{s=1}^n c_{ij}^s\mathfrak{f}_{sm}\mid\, i,j,m\leq n \,\rangle + \langle\,\det(F)^{q-1}-1\,\rangle,$$
where $F$ denotes the matrix of entries in $\mathfrak{F}_n$. Besides,
$|\mathcal{V}(I^{\mathrm{Isom}}_{A,A'})|= \mathrm{dim}_{\mathbb{F}_q}(\mathbb{F}_q[\mathfrak{F}_n]/ I^{\mathrm{Isom}}_{A,A'})$.
\end{enumerate}
\end{thm}

\begin{proof} Let us prove the second assertion, being analogous the reasoning for assertion (a). The generators of the ideal $I^{\mathrm{Isom}}_{A,A'}$ involve each zero $(f_{11},\ldots,$ $f_{nn})$ of its algebraic set to constitute the entries of the regular matrix of an isomorphism $f$ between the algebras $A$ and $A'$. The result follows from the fact of being this ideal zero-dimensional and radical. Particularly, the ideal $I^{\mathrm{Isom}}_{A,A'}$ is zero-dimensional because its algebraic set is a finite subset of $\mathbb{F}_q^{n^2}$. Besides, from Proposition 2.7 of \cite{Cox1998}, the ideal $I$ is also radical, because, for each $i,j\leq n$, the unique monic generator of $I\cap \mathbb{F}_q[\mathfrak{f}_{ij}]$ is the polynomial $(\mathfrak{f}_{ij})^q-\mathfrak{f}_{ij}$, which is intrinsically included in each ideal of $\mathbb{F}_q[\mathfrak{F}_n]$ and is square-free.
\end{proof}

\vspace{0.2cm}

\begin{corollary}\label{coro_CAG_Isom} The complexity times that Buchberger's algorithm requires to compute the reduced Gr\"obner bases of the ideals $I^{\mathrm{Isot}}_{A,A'}$ and $I^{\mathrm{Isom}}_{A,A'}$ in Theorem \ref{thm_CAG_Isom} are, respectively, $q^{O(3n^2)}+O(n^6n!)$ and $q^{O(n^2)}+O(n^6n!)$.
\end{corollary}

\begin{proof} We prove the result for the second ideal, being analogous the reasoning for the first one. The result follows straightforwardly from Theorem \ref{Gao} once we observe that all the generators of the ideal in Theorem \ref{thm_CAG_Isom} are sparse in $\mathbb{F}_q[\mathfrak{F}_n]$. More specifically, the number of variables is $n^2$, the number of generators of the ideal under consideration that are not of the form $(\mathfrak{f}_{ij})^q-\mathfrak{f}_{ij}$ is $n^3+1$ and the maximal length of these generators is $n!$.
\end{proof}

\vspace{0.2cm}

Theorem \ref{thm_CAG_Isom} has been implemented as a procedure called {\em isoAlg} in the open computer algebra system for polynomial computations {\sc Singular} \cite{Decker2016}. This has been included in the library {\em GraphAlg.lib}, which is available online at {\texttt{http://personales.us.es/raufalgan/LS/GraphAlg.lib}}. Let us illustrate the use of this procedure with an example related to the distribution of the set $\mathcal{P}_2(\mathbb{F}_2)$ of two-dimensional partial quasigroup rings over the finite field $\mathbb{F}_2$ into isotopism and isomorphism classes. All the computations that are exposed throughout this paper are implemented in a system with an {\em Intel Core i7-2600, with a 3.4 GHz processor and 16 GB of RAM}.

\begin{example}\label{ejemplo_PQ2} Let us consider the pair of partial quasigroup rings in $\mathcal{P}_2(\mathbb{F}_2)$ that are respectively related to the partial Latin squares
$$\begin{array}{|c|c|} \hline 1 & 2\\ \hline
2 & \ \\ \hline \end{array}  \hspace{0.25cm} \text { and } \hspace{0.25cm} \begin{array}{|c|c|} \hline  1 & 2\\ \hline
2 & 1 \\ \hline \end{array}$$
These two partial Latin squares are not isotopic because isotopisms preserve the number of filled cells. Nevertheless, their related partial quasigroup rings over $\mathbb{F}_2$, with respective bases $\{e_1,e_2\}$ and $\{e'_1,e'_2\}$, and which are respectively described by the products
$$\begin{cases}e_1e_1=e_1,\\
e_1e_2=e_2=e_2e_1.
\end{cases} \hspace{0.5cm}  \text{ and } \hspace{0.5cm} \begin{cases}e'_1e'_1=e'_1=e'_2e'_2,\\
e'_1e'_2=e'_2=e'_2e'_1.
\end{cases}$$
are isotopic. Specifically, by implementing the procedure {\em isoAlg}, our system computes in $0$ seconds the existence of four isotopisms between these two partial quasigroup rings. One of this isotopisms is, for instance, the isomorphism $f$ such that $f(e_1)=e'_1$ and $f(e_2)=e'_1+e'_2$. The procedure {\em isoAlg} also ensures us the existence of $f$ as the unique possible isomorphism. \hfill $\lhd$
\end{example}

\vspace{0.2cm}

In practice, in those cases in which the run time required for the computations involved in Theorem \ref{thm_CAG_Isom} becomes excessive, it is recommendable to eliminate the generators of the corresponding ideal that are referred to the determinants of the matrices $F$, $G$ and $H$. This reduces the complexity time in Corollary \ref{coro_CAG_Isom} to $q^{O(3n^2)}  + O(n^8)$ and $q^{O(n^2)} + O(n^8)$, respectively, and gives enough information to analyze a case study on which base the possible isomorphisms and isotopisms between two given algebras, whatever the base field is. The next example illustrates this fact by focusing on the possible isotopisms that there exist over any field between the two partial quasigroup rings that appear in Example \ref{ejemplo_PQ2}.

\begin{example}\label{ejemplo_PQ2a} The implementation of the procedure {\em isoAlg} enables us to ensure that, whatever the base field is, the reduced Gr\"obner basis of the ideal $I^{\mathrm{Isot}}_{A,A'}$ in Theorem \ref{thm_CAG_Isom} related to the isotopism group between the two partial quasigroup rings of Example \ref{ejemplo_PQ2} holds that
$2\mathfrak{h}_{22}^3=0$ and $\mathfrak{h}_{21}^2+\mathfrak{h}_{22}^2=0$. If the characteristic of the base field is not two, then $\mathfrak{h}_{21}=\mathfrak{h}_{22}=0$. This involves $H$ to be singular and hence, these two partial quasigroup rings are not isotopic. Otherwise, it is straightforwardly verified that the linear transformation $f$ that is indicated in Example \ref{ejemplo_PQ2} constitutes an isomorphism between both rings for
every base field of characteristic two. \hfill $\lhd$
\end{example}

\section{Description of faithful functors between algebras and graphs}

Based on the proposal of McKay et al. \cite{McKay2007} for Latin squares, we describe now a pair of graphs that are uniquely related to a finite-dimensional algebra $A$ over a finite field $\mathbb{K}$. Firstly, we define the vertex-colored graph $G_1(A)$ with four maximal monochromatic subsets $R_{A}=\{r_u\mid\, u\in A\setminus \mathrm{Ann}_{A^-}(A)\}$, $C_{A}=\{c_u\mid\, u\in A\setminus \mathrm{Ann}_{A^+}(A)\}$, $S_{A}=\{s_u\mid\, u\in A^2\setminus \{0\}\}$ and $T_{A}=\{t_{u,v}\mid\, u,v\in A, uv\neq 0\}$, and edges $\{r_ut_{u,v}, c_vt_{u,v}, s_{uv}t_{u,v}\mid u,v\in A, uv\neq 0\}$. From this graph we also define the vertex-colored graph $G_2(A)$ by adding the edges $\{r_uc_u,\mid\, u\in A\setminus \mathrm{Ann}_A(A)\}\, \cup \{c_us_u\mid\, u \in A^2\setminus \mathrm{Ann}_{A^+}(A)\} \, \cup \{r_us_u\mid\, u \in A^2\setminus \mathrm{Ann}_{A^-}(A)\}$. As an illustrative example, Figure \ref{Fig_1} shows the two graphs that are related to any $n$-dimensional algebra over the finite field $\mathbb{F}_2$, with basis $\{e_1,\ldots,e_n\}$, that is described as $e_1e_2=e_2e_1=e_1$.

\begin{figure}[h]
\begin{center}
$\begin{array}{c|c}
G_1 & G_2 \\ \hline
\includegraphics[width=6cm]{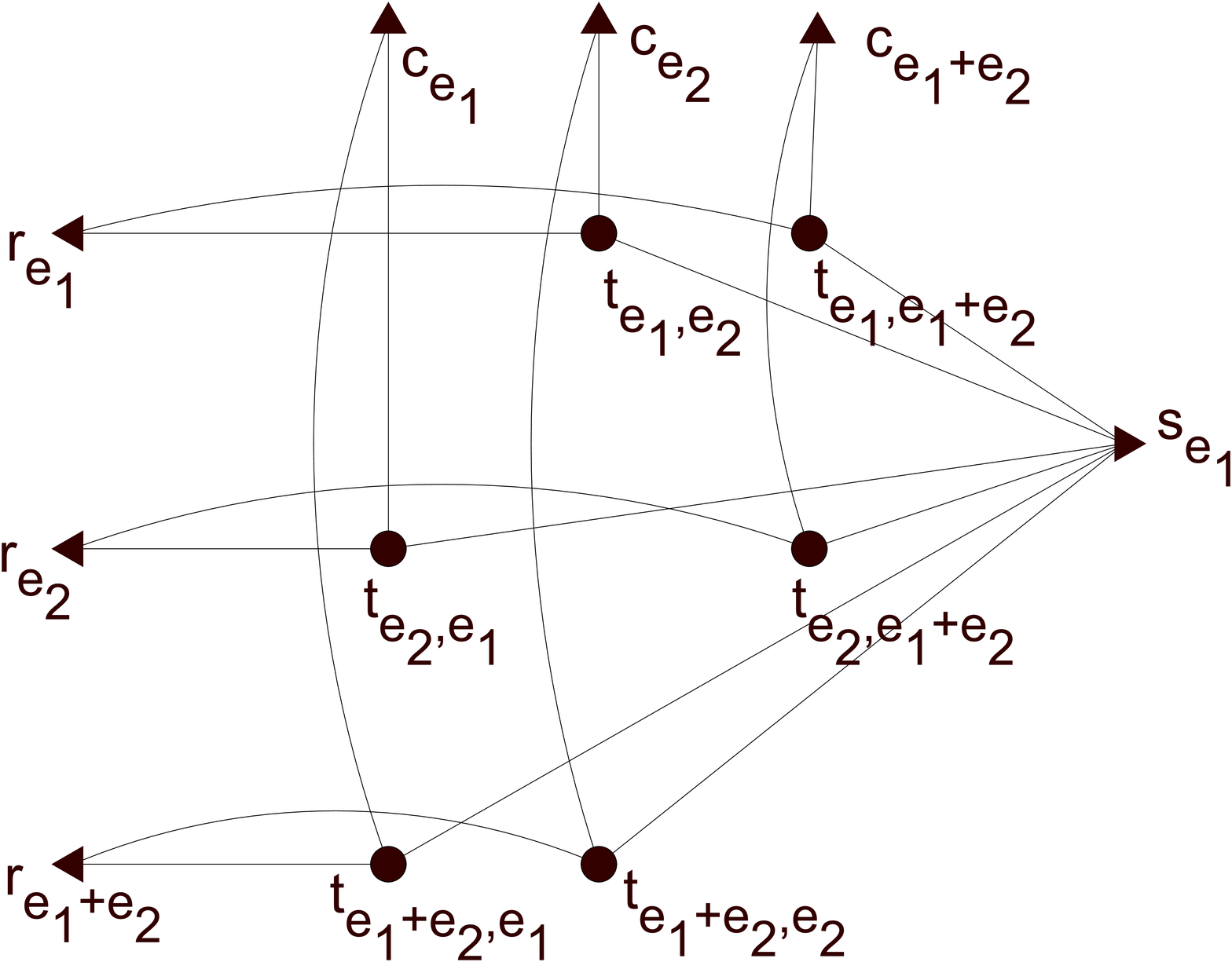} & \includegraphics[width=6cm]{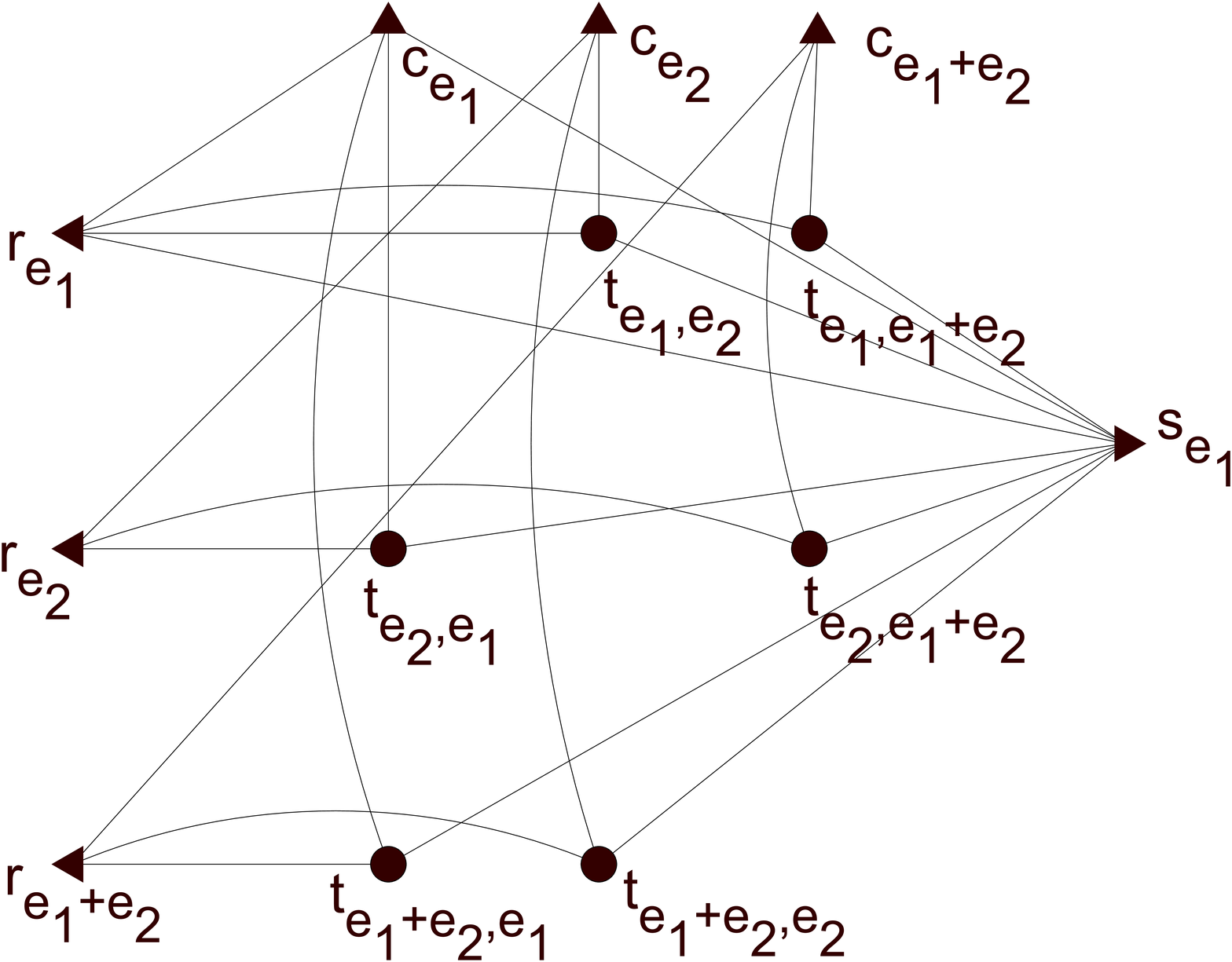}
\end{array}$
\end{center}
\caption{Graphs related to the algebra $e_1e_2=e_2e_1=e_1$ over $\mathbb{F}_2$.}
\label{Fig_1}
\end{figure}

\begin{lemm}\label{lemm_graph0} The next assertions hold.
\begin{enumerate}[a)]
\item If the algebra $A$ is abelian, then $G_1(A)$ and $G_2(A)$ have no vertices.
\item The graph $G_1(A)$ does not contain triangles.
\item In both graphs $G_1(A)$ and $G_2(A)$,
\begin{itemize}
\item The number of vertices is
{\small \[|A\setminus \mathrm{Ann}_{A^-}(A)|+|A\setminus \mathrm{Ann}_{A^+}(A)|+
|A^2|+|\{(u,v)\in A\times A\mid\, uv\neq 0\}| - 1.
\]}
\item The degree of the vertex $t_{u,v}$ is
\[d(t_{u,v})=3, \text{ for all } u,v\in A \text{ such that } uv\neq 0.\]
\end{itemize}
\item In the graph $G_1(A)$,
\begin{itemize}
\item $d(r_u)=|A\setminus \mathrm{Ann}_{A^+}(\{u\})|$, for all $u\not\in \mathrm{Ann}_{A^-}(A)$.
\item $d(c_u)=|A\setminus \mathrm{Ann}_{A^-}(\{u\})|$, for all $u\not\in \mathrm{Ann}_{A^+}(A)$.
\item $d(s_u)=\sum_{v\in A}|\mathrm{ad}^{-1}_v(u)|$, for all $u\in A^2\setminus\{0\}$, where $\mathrm{ad}_v:A\rightarrow A^2$ is the {\em adjoint action} of $v$ in $A$ such that $\mathrm{ad}_v(w)=vw$, for all $w\in A$.
\end{itemize}
\item Let $\mathbf{1}$ denotes the characteristic function. Then, in the graph $G_2(A)$,
\begin{itemize}
\item $d(r_u)=|A\setminus \mathrm{Ann}_{A^+}(\{u\})|+ \mathbf{1}_{A\setminus\mathrm{Ann}_{A^+}(A)}(u) + \mathbf{1}_{A^2}(u)$, for all $u\not\in \mathrm{Ann}_{A^-}(A)$.
\item $d(c_u)=|A\setminus \mathrm{Ann}_{A^-}(\{u\})|+     \mathbf{1}_{A\setminus\mathrm{Ann}_{A^-}(A)}(u) + \mathbf{1}_{A^2}(u)$, for all $u\not\in \mathrm{Ann}_{A^+}(A)$.
\item $d(s_u)=\mathbf{1}_{A\setminus\mathrm{Ann}_{A^-}(A)}(u)+ \mathbf{1}_{A\setminus\mathrm{Ann}_{A^+}(A)}(u)+\sum_{v\in A}|\mathrm{ad}^{-1}_v(u)|$, for all $u\in A^2\setminus\{0\}$.
\end{itemize}
\end{enumerate}
\end{lemm}

\begin{proof} The result follows straightforwardly from the definition of the graphs $G_1(A)$ and $G_2(A)$.
\end{proof}

\vspace{0.2cm}

\begin{prp}\label{prop_graph0_a} The next assertions hold.
\begin{enumerate}[a)]
\item The number of edges of the graph $G_1(A)$ is
$\sum_{u\not\in\mathrm{Ann}_{A^-}(A)} |A\setminus \mathrm{Ann}_{A^+}(\{u\})|$ $+\sum_{u\not\in\mathrm{Ann}_{A^+}(A)}|A\setminus \mathrm{Ann}_{A^-}(\{u\})|+\sum_{u\in A^2\setminus\{0\}}\sum_{v\in A}|\mathrm{ad}_v^{-1}(u)|$.
\item The number of edges of the graph $G_2(A)$ coincides with those of $G_1(A)$ plus $|A\setminus\mathrm{Ann}_A(A)| + |A^2\setminus \mathrm{Ann}_{A^-}(A)| + |A^2\setminus\mathrm{Ann}_{A^+}(A)|$.
\end{enumerate}
\end{prp}

\begin{proof} The result follows from the first theorem of Graph Theory \cite{Harary1969}, which enables us to ensure that the number of edges of a graph is the half of the summation of degrees of its vertices. Now, for each pair of vectors $u,v\in A$ such that $uv\neq 0$, the vertex $t_{u,v}\in T_A$ is the only vertex in $T_A$ that is adjacent to the vertices $r_u\in R_A$, $c_v\in C_A$ and $s_{uv}\in S_A$. They constitute indeed the three vertices related to the degree of $t_{u,v}$ that is indicated in assertion (c) of Lemma \ref{lemm_graph0}. As a consequence, the summation of degrees of all the vertices in $T_A$ coincides with $\sum_{u\in R_A}d(r_u) + \sum_{u\in C_A}d(c_u) + \sum_{u\in S_A}d(s_u)$. The result follows then from assertions (d) and (e) in Lemma \ref{lemm_graph0}.
\end{proof}

\vspace{0.2cm}

\begin{thm}\label{theo_graph0} Let $A$ and $A'$ be two finite-dimensional algebras over a finite field $\mathbb{K}$. Then,
\begin{enumerate}[a)]
\item If both algebras are isotopic, then their corresponding graphs $G_1(A)$ and $G_1(A')$ are isomorphic. Reciprocally, if the graphs $G_1(A)$ and $G_1(A')$ are isomorphic, then there exist three bijective maps $f$, $g$ and $h$ between $A$ and $A'$ such that $f(u)g(v)=h(uv)$.
\item If both algebras are isomorphic, then their corresponding graphs $G_2(A)$ and $G_2(A')$ are also isomorphic. Reciprocally, if the graphs $G_2(A)$ and $G_2(A')$ are isomorphic, then there exists a multiplicative bijective map between the algebras $A$ and $A'$, that is, a bijective map $f:A\rightarrow A'$ so that $f(u)f(v)=f(uv)$, for all $u,v\in A$.
\end{enumerate}
\end{thm}

\begin{proof} Let $(f,g,h)$ be an isotopism between the algebras $A$ and $A'$. We define the map $\alpha$ between $G_1(A)$ and $G_1(A')$ such that
$$\begin{cases}\alpha(r_u)=r_{f(u)}, \text{ for all } u\in  A\setminus \mathrm{Ann}_{A^-}(A),\\
\alpha(c_u)=c_{g(u)}, \text{ for all } u\in  A\setminus \mathrm{Ann}_{A^+}(A),\\
\alpha(s_u)=s_{h(u)}, \text{ for all } u\in  A^2\setminus \{0\},\\
\alpha(t_{u,v})=t_{f(u),g(v)}, \text{ for all } u,v\in A \text { such that } uv\neq 0.\end{cases}$$

The description of $G_1(A)$ and $G_1(A')$, together with Lemmas \ref{lemm_annihilator} and \ref{lemm_M1}, and the regularity of $f$, $g$ and $h$, involves $\alpha$ to be an isomorphism between these two graphs. The same map $\alpha$ constitutes an isomorphism between the graphs $G_2(A)$ and $G_2(A')$ in case of being $f=g=h$, that is, if the algebras $A$ and $A'$ are isomorphic. Reciprocally, let $\alpha$ be an isomorphism between the graphs $G_1(A)$ and $G_1(A')$. Collinearity involves this isomorphism to be uniquely determined by its restriction to $R_A\cup C_A\cup S_A$. Specifically, the image of each vertex $t_{u,v}\in T_A$ by means of $\alpha$ is uniquely determined by the corresponding images of $r_u$, $c_v$ and $s_{uv}$. Let $\beta$ and $\beta'$ be the respective bases of the algebras $A$ and $A'$ and let $\pi: A\rightarrow A'$ be the natural map that preserves the components of each vector with respect to the mentioned bases. That is, $\pi((u_1,\ldots,u_n)_{\beta})=(u_1,\ldots,u_n)_{\beta'}$, for all $u_1,\ldots,u_n\in\mathbb{K}$. Let us define three maps $f$, $g$ and $h$ from $A$ to $A'$ such that
$$f(u)=\begin{cases}\pi(u), \text{ for all } u\in \mathrm{Ann}_{A^-}(A),\\
v, \text{ otherwise, where } v\in A \text{ is such that } \alpha(r_u)=r_v.
\end{cases}$$
$$g(u)=\begin{cases}\pi(u), \text{ for all } u\in \mathrm{Ann}_{A^+}(A),\\
v, \text{ otherwise, where } v\in A \text{ is such that } \alpha(c_u)=c_v.
\end{cases}$$
$$h(u)=\begin{cases}\pi(u), \text{ for all } u\in (A\setminus A^2)\cup \{0\},\\
v, \text{ otherwise, where } v\in A \text{ is such that } \alpha(s_u)=s_v.
\end{cases}$$

\vspace{0.1cm}

From Lemmas \ref{lemm_annihilator} and \ref{lemm_M1}, these three maps are bijective. Let $u,v\in A$. If $u\in \mathrm{Ann}_{A^-}(A)$ or $v\in \mathrm{Ann}_{A^+}(A)$, then there does not exist the vertex $t_{u,v}$ in the graph $G_1(A)$. Since $\alpha$ preserves collinearity, there does not exist the vertex $t_{f(u),g(v)}$ in the graph $G_1(A')$, which means that $f(u)\in \mathrm{Ann}_{{A'}^-}(A')$ or $g(v)\in \mathrm{Ann}_{{A'}^+}(A')$. In any case, we have that $f(u)g(v)=0=h(uv)$. Finally, if $u\not\in\mathrm{Ann}_{A^-}(A)$ and $v\not \in \mathrm{Ann}_{A^+}(A)$, then the vertex $t_{u,v}$ connects the vertices $r_u$, $c_v$ and $s_{uv}$ in the graph $G_1(A)$. Now, the isomorphism $\alpha$ maps this vertex $t_{u,v}$ in $G_1(A)$ to a vertex $t_{u',v'}$ in $G_2(A)$ that is connected to the vertices $r_{u'}$, $c_{v'}$ and $s_{u'v'}$. Again, since $\alpha$ preserves collinearity, it is $f(u)=u'$, $g(v)=v'$ and, finally, $h(uv)=f(u)g(v)$.

\vspace{0.1cm}

In case of being $\alpha$ an isomorphism between the graphs $G_2(A)$ and $G_2(A')$ it is enough to consider $f=g=h$ in the previous description. This is well-defined because of the new edges that are included to the graphs $G_1(A)$ and $G_1(A')$ in order to define, respectively, the graphs $G_2(A)$ and $G_2(A')$. These edges involve the multiplicative character of the bijective map $f$, that is, $f(u)g(v)=h(uv)$, for all $u,v\in A$.
\end{proof}

\vspace{0.1cm}

Theorem \ref{theo_graph0} enables us to ensure that graph invariants reduce the cost of computation that is required to distribute a set of finite-dimensional algebras over finite fields into isotopism and isomorphism classes. It is only necessary to compute those reduced Gr\"obner bases in Theorem \ref{thm_CAG_Isom} that are related to a pair of algebras whose associated graphs have equal invariants. The complexity to compute these invariants is always much less than that related to the calculus of a reduced Gr\"obner basis. Thus, for instance, the complexity to compute the number of vertices, edges and triangles of the graphs related to any $n$-dimensional algebra over the finite field $\mathbb{F}_q$ is $q^{O(2n)}$. This corresponds to the computation of the adjacency matrices of both graphs by means of all the possible products among the $q^n$ distinct vectors of any such an algebra. This enables us in particular to implement the formulas exposed in Lemma \ref{lemm_graph0} and Proposition \ref{prop_graph0_a}. Besides, recall that the trace of the adjacency matrix of a graph raised to the third power coincides with the number of triangles of such a graph. All this computation has been implemented in the procedure {\em isoGraph}, which has been included in the mentioned library {\em GraphAlg.lib}. In order to illustrate the efficiency of these invariants, we focus on the set $\mathcal{L}_n(\mathbb{F}_q)$ of $n$-dimensional Lie algebras over the finite field $\mathbb{F}_q$, with $q$ a power prime. Recall that a {\em Lie algebra} is an anti-commutative algebra $A$ that holds the {\em Jacobi identity} $u(vw)+v(wu)+w(uv)=0$, for all $u,v,w\in A$. For $n=3$, it is known \cite{Falcon2016a, Graaf2005, Strade2007} that there are $32$ distinct Lie algebras in $\mathcal{L}_3(\mathbb{F}_2)$, which are distributed into four isotopism classes and six isomorphism classes, and $123$ Lie algebras in $\mathcal{L}_3(\mathbb{F}_3)$, which are distributed into four isotopism classes and seven isomorphism classes. Table \ref{Table_0} shows the run time and memory usage that our computer system requires to determine the mentioned classification depending on whether graph invariants are considered ({\em Graph}) or not ({\em Alg}). Further, Tables \ref{Table_1} and \ref{Table_1a} show, respectively, the invariants of the graphs $G_1$ and $G_2$ related to the isomorphism classes of $\mathcal{L}_3(\mathbb{F}_q)$, for $q\in\{2,3\}$. The components of the $4$-tuples that are indicated in the corresponding columns of vertices in both tables refer, respectively, to the number of vertices in $R_A$, $C_A$, $S_A$ and $T_A$.

\begin{table}[ht]
\begin{center}
\resizebox{\textwidth}{!}{
\begin{tabular}{c|cc|cc|cc}
\ & \multicolumn{2}{c|}{Graph} & \multicolumn{4}{c}{Alg}\\
\ & & & \multicolumn{2}{c|}{Isomorphisms}  & \multicolumn{2}{c}{Isotopisms}\\
$q$ &  Run time  & Memory usage & Run time & Memory usage & Run time & Memory usage\\ \hline
2 & 1 s & 0 Mb & 1 s & 0 Mb & 34 s & 384 Mb\\
3 & 47 s & 3 Mb & 4 s & 6 Mb & \multicolumn{2}{c}{Run out of memory}\\
\end{tabular}}
\caption{Cost of computation to distribute $\mathcal{L}_3(\mathbb{F}_q)$ into isotopism and isomorphism classes.}
\label{Table_0}
\end{center}
\end{table}

\begin{table}[ht]
\resizebox{\textwidth}{!}{
\begin{tabular}{c|cc|cc}
\ & \multicolumn{2}{c|}{$\mathbb{F}_2$} & \multicolumn{2}{c}{$\mathbb{F}_3$} \\
$ A$ & Vertices & Edges & Vertices & Edges \\ \hline
Abelian & (0,0,0,0) & 0 & (0,0,0,0) & 0\\
$e_1e_2=e_3$ & (6,6,1,24) & 72 & (24,24,2,432) & 1296\\
$e_1e_2=e_1$ & (6,6,1,24) & 72 & (24,24,2,432) & 1296\\
$e_1e_2=e_3, e_1e_3=-e_2$ & - & - & (26,26,8,576) & 1728 \\
$e_1e_2=e_3, e_1e_3=e_2$ & (7,7,3,36) & 108 & (26,26,8,576) & 1728\\
$e_1e_2=e_2, e_1e_3=e_3$ & (7,7,3,36) & 108 & (26,26,8,576) & 1728\\
$e_1e_2=e_2, e_1e_3=-e_3, e_2e_3=-e_1$ & (7,7,7,42) & 126 & - & -\\
$e_1e_2=e_2, e_1e_3=-e_3, e_2e_3=2e_1$ & - & - & (26,26,26,624) & 1872\\
\end{tabular}}
\caption{Invariants of the graph $G_1$ related to $\mathcal{L}_3(\mathbb{F}_q)$, for $q\in\{2,3\}$.}
\label{Table_1}
\end{table}

\begin{table}[ht]
\resizebox{\textwidth}{!}{
\begin{tabular}{c|ccc|ccc}
\ & \multicolumn{3}{c|}{$\mathbb{F}_2$} & \multicolumn{3}{c}{$\mathbb{F}_3$} \\
$A$ & Vertices & Edges & Triangles & Vertices & Edges & Triangles\\ \hline
Abelian & (0,0,0,0) & 0 & 0 & (0,0,0,0) & 0 & 0\\
$e_1e_2=e_3$ & (6,6,1,24) & 78 & 0 & (24,24,2,432) & 1320 & 0\\
$e_1e_2=e_1$ & (6,6,1,24) & 80 & 9 & (24,24,2,432) & 1324 & 38\\
$e_1e_2=e_3, e_1e_3=-e_2$ & - & - & - & (26,26,8,576) & 1770 & 8\\
$e_1e_2=e_3, e_1e_3=e_2$ & (7,7,3,36) & 121 & 11 & (26,26,8,576) & 1770 & 80\\
$e_1e_2=e_2, e_1e_3=e_3$ & (7,7,3,36) & 121 & 27 & (26,26,8,576) & 1770 & 152\\
$e_1e_2=e_2, e_1e_3=-e_3, e_2e_3=-e_1$ & (7,7,7,42) & 147 & 19 & - & - & -\\
$e_1e_2=e_2, e_1e_3=-e_3, e_2e_3=2e_1$ & - & - & - & (26,26,26,624) & 1950 & 74\\
\end{tabular}}
\caption{Invariants of the graph $G_2$ related to $\mathcal{L}_3(\mathbb{F}_q)$, for $q\in\{2,3\}$.}
\label{Table_1a}
\end{table}

\section{Graphs and partial quasigroup rings}

Let us finish with an illustrative example that focuses on those graphs $G_1$ and $G_2$ related to the set $\mathcal{P}_n(\mathbb{K})$ of $n$-dimensional partial quasigroup rings over a finite field $\mathbb{K}$ that are based on the known distribution of partial Latin squares of order $n\leq 3$ into isotopism classes. In this regard, Table \ref{Table_PQ} shows several graph invariants that are related to the isotopism classes of $\mathcal{P}_2(\mathbb{F}_q)$, for $q\in\{2,3\}$. Partial Latin squares are written row after row in a single line, with empty cells represented by zeros.

\begin{table}[ht]
\begin{center}
\resizebox{\textwidth}{!}{
\begin{tabular}{c|c|c|cc|c|c|cc}
\ &  \multicolumn{4}{c|}{$\mathbb{F}_2$} &  \multicolumn{4}{c}{$\mathbb{F}_3$}\\
\ & $G_1\, \& \, G_2$ & $G_1$ & \multicolumn{2}{c|}{$G_2$} & $G_1\, \& \, G_2$ & $G_1$ & \multicolumn{2}{c}{$G_2$}\\
Partial Latin square & Vertices & Edges & Edges & Triangles & Vertices & Edges & Edges & Triangles\\ \hline
00 00 & (0,0,0,0) & 0 & 0 & 0 & (0,0,0,0) & 0 & 0 & 0\\
10 00 & (2,2,1,4) & 12 & 16 & 7 & (6,6,2,36) & 108 & 118 & 20\\
10 01 & (3,3,1,6) & 18 & 23 & 7 & (8,8,2,48) & 144 & 156 & 22\\
10 02 & (3,3,3,7) & 21 & 30 & 16 & (8,8,8,56) & 168 & 192 & 48\\
10 20 & (3,2,3,6) & 18 & 25 & 12 & (8,6,8,48) & 144 & 164 & 42\\
12 00 & (2,3,3,6) & 18 & 25 & 12 & (6,8,8,48) & 144 & 164 & 42\\
12 20 & (3,3,3,8) & 24 & 33 & 13 & (8,8,8,60) & 180 & 204 & 38\\
12 21 & (3,3,3,8) & 24 & 33 & 13 & (8,8,8,56) & 168 & 192 & 48\\
\end{tabular}}
\caption{Invariants of the graphs $G_1$ and $G_2$ related to $\mathcal{P}_2(\mathbb{F}_q)$, for $q\in\{2,3\}$.}
\label{Table_PQ}
\end{center}
\end{table}

\begin{thm}\label{theo_PQ2a} The set $\mathcal{P}_2(\mathbb{K})$ is distributed into seven isotopism classes, whatever the base field is.
\end{thm}

\begin{proof} A computational case study based on a similar reasoning to that exposed in Example \ref{ejemplo_PQ2a} enables us to ensure the result. If the characteristic of the base field is distinct from two, then the seven isotopism classes under consideration are those related to the next partial Latin squares of order $2$
$$\begin{array}{|c|c|} \hline
\ & \ \\ \hline
\ & \ \\ \hline \end{array}  \hspace{0.5cm}
\begin{array}{|c|c|} \hline
1 & \ \\ \hline
\ & \ \\ \hline \end{array}  \hspace{0.5cm}
\begin{array}{|c|c|} \hline
1 & \ \\ \hline
\ & 1 \\ \hline \end{array} \hspace{0.5cm}
\begin{array}{|c|c|} \hline
1 & \ \\ \hline
2 & \ \\ \hline \end{array} \hspace{0.5cm}
\begin{array}{|c|c|} \hline
1 & 2 \\ \hline
\ & \ \\ \hline \end{array} \hspace{0.5cm}
\begin{array}{|c|c|} \hline
1 & 2 \\ \hline
2 & \ \\ \hline \end{array}\hspace{0.5cm}
\begin{array}{|c|c|} \hline
1 & 2 \\ \hline
2 & 1 \\ \hline \end{array}$$

Otherwise, if the characteristic of the base field is two, then the isotopism classes related to the last two partial Latin squares coincide. In this case, the next partial Latin square corresponds to the seventh isotopism class
$$\begin{array}{|c|c|} \hline
1 & \ \\ \hline
\ & 2 \\ \hline \end{array}$$

If the characteristic of the base field is distinct from two, the partial quasigroup ring related to this partial Latin square is isotopic to that related to the unique Latin square of the previous list.
\end{proof}

\vspace{0.3cm}

It is known \cite{Falcon2013} that there are $2$, $8$ and $81$ distinct isotopism classes of partial Latin squares of respective orders $1$ to $3$. In order to determine those distinct isotopism classes that give rise to isotopic partial quasigroup rings over the finite field $\mathbb{F}_2$, we have implemented the procedure {\em isoAlg} in our previously mentioned computer system. With a total run time of $761$ seconds, we have obtained that there exist $2$, $7$ and $72$ distinct isotopism classes of partial quasigroup rings of respective dimensions $1$ to $3$. Particularly, the existence of two classes for the one-dimensional case agrees with that exposed in Subsection 2.3. Besides, the seven isotopism classes for the two-dimensional case agrees with Theorem \ref{theo_PQ2a}. For the three-dimensional case, the next nine pairs of non-isotopic partial Latin squares give rise to isotopic partial quasigroup rings
$${\scriptsize\begin{array}{ccccc}
\begin{array}{|c|c|c|} \hline
1 & 2 & \ \\ \hline
2 & \ & \ \\ \hline
\ & \ & \ \\ \hline \end{array}  \text { and } \begin{array}{|c|c|c|} \hline
1 & 2 & \ \\ \hline
2 & 1 & \ \\ \hline
\ & \ & \  \\ \hline \end{array} & , &
\begin{array}{|c|c|c|} \hline
1 & 2 & \ \\ \hline
2 & \ & \ \\ \hline
\ & \ & 1 \\ \hline \end{array}  \text { and } \begin{array}{|c|c|c|} \hline
1 & 2 & \ \\ \hline
2 & 1 & \ \\ \hline
\ & \ & 1 \\ \hline \end{array} & , & \begin{array}{|c|c|c|} \hline
1 & 2 & \ \\ \hline
2 & \ & \ \\ \hline
\ & \ & 3 \\ \hline \end{array}  \text { and } \begin{array}{|c|c|c|} \hline
1 & 2 & \ \\ \hline
2 & 1 & \ \\ \hline
\ & \ & 3  \\ \hline \end{array}\\ \\
\begin{array}{|c|c|c|} \hline
1 & 2 & \ \\ \hline
\ & 1 & \ \\ \hline
3 & \ & \ \\ \hline \end{array} \text { and }
\begin{array}{|c|c|c|} \hline
1 & 2 & \ \\ \hline
2 & 1 & \ \\ \hline
3 & \ & \  \\ \hline \end{array} &,&\begin{array}{|c|c|c|} \hline
1 & 2 & \ \\ \hline
\ & 1 & 3 \\ \hline
\ & \ & \ \\ \hline \end{array} \text { and }
\begin{array}{|c|c|c|} \hline
1 & 2 & 3 \\ \hline
2 & 1 & \ \\ \hline
\ & \ & \  \\ \hline \end{array} &,&
\begin{array}{|c|c|c|} \hline
1 & 2 & \ \\ \hline
\ & 1 & \ \\ \hline
3 & \ & 2 \\ \hline \end{array}  \text { and } \begin{array}{|c|c|c|} \hline
1 & 2 & \ \\ \hline
2 & 1 & \ \\ \hline
3 & \ & 1 \\ \hline \end{array}\\ \\
\begin{array}{|c|c|c|} \hline
1 & 2 & \ \\ \hline
2 & \ & 3 \\ \hline
\ & \ & 1 \\ \hline \end{array}   \text { and } \begin{array}{|c|c|c|} \hline
1 & 2 & \ \\ \hline
2 & 1 & 3 \\ \hline
\ & \ & 1 \\ \hline \end{array} &,&
\begin{array}{|c|c|c|} \hline
1 & 2 & \ \\ \hline
\ & 1 & 3 \\ \hline
3 & \ & \ \\ \hline \end{array} \text { and } \begin{array}{|c|c|c|} \hline
1 & 2 & \ \\ \hline
2 & 1 & 3 \\ \hline
3 & \ & \ \\ \hline \end{array}&,&
\begin{array}{|c|c|c|} \hline
1 & 2 & \ \\ \hline
\ & 1 & 3 \\ \hline
3 & \ & 2 \\ \hline \end{array} \text { and } \begin{array}{|c|c|c|} \hline
1 & 2 & 3 \\ \hline
2 & 1 & \ \\ \hline
3 & \ & 1 \\ \hline \end{array}
\end{array}}$$

\vspace{0.2cm}

The run time of 761 seconds that are required to determine the mentioned distribution of partial quasigroup rings reduces to only 30 seconds in the same computer system if the invariants that we have just exposed in Table \ref{Table_PQ} and those exposed in Table \ref{Table_PQ3} are previously computed. The new run time includes indeed the extra 9 seconds of computation that is required for computing such invariants.

\begin{table}[ht]
\resizebox{\textwidth}{!}{
\begin{tabular}{ccc|ccc|ccc}
Partial Latin square & Vertices & Edges & Partial Latin square & Vertices & Edges & Partial Latin square & Vertices & Edges\\ \hline
100 000 000	& (4,4,1,16) & 48  & 100 010 002 & (7,7,3,34)  & 120 & 031 302	& (7,7,7,42) & 126	\\
120 000 000	& (4,6,3,24) & 72  & 120 001 002	& (7,7,3,36) & 108 & 120 210 301	 & (7,7,7,42) & 126	\\
123 000 000	& (4,7,7,28) & 84  & 120 200 002	& (7,7,3,36) & 108 & 120 213 001	 & (7,7,7,42) & 126	\\
100 200 000	& (6,4,3,24) & 72 & 120 200 001	& (7,7,3,38) & 114 & 120 213 300	 & (7,7,7,42) & 126	\\
100 010 000	& (6,6,1,24) & 72 & 120 210 001	& (7,7,3,38) & 114 & 120 001 312	 & (7,7,7,43) & 129	\\
100 020 000	& (6,6,3,28) & 84 & 120 201 010	& (7,7,3,40) & 120 & 120 201 302	 & (7,7,7,43) & 129	\\
120 200 000	& (6,6,3,32) & 96 & 120 201 012	& (7,7,3,40) & 120 & 120 231 300	 & (7,7,7,43) & 129	\\
120 210 000	& (6,6,3,32) & 96 & 100 020 003	& (7,7,7,37) & 111 & 123 231 312	 & (7,7,7,43) & 129	\\
120 000 300	& (6,6,6,32) & 96 & 120 002 003	& (7,7,7,38) & 114 & 120 003 312	 & (7,7,7,44) & 132	\\
120 000 310	& (6,6,6,36) & 108 & 120 002 300	& (7,7,7,38) & 114 & 120 013 301	 & (7,7,7,44) & 132	\\
120 001 000	& (6,7,3,32) & 96 & 120 003 300	& (7,7,7,38) & 114 & 120 013 302	 & (7,7,7,44) & 132	\\
120 012 000	& (6,7,3,36) & 108 & 120 001 300	& (7,7,7,39) & 117 & 120 200 312	 & (7,7,7,44) & 132	\\
120 003 000	& (6,7,7,34) & 102 & 120 200 003	& (7,7,7,40) & 120 & 120 203 301	 & (7,7,7,44) & 132	\\
120 000 302	& (6,7,7,36) & 108 & 120 200 302	& (7,7,7,40) & 120 & 123 210 301	 & (7,7,7,44) & 132	\\
123 200 000	& (6,7,7,36) & 108 & 120 210 003	& (7,7,7,40) & 120 & 123 031 310	 & (7,7,7,45) & 135	\\
120 013 000	& (6,7,7,38) & 114 & 123 010 001	& (7,7,7,40) & 120 & 123 200 312	 & (7,7,7,45) & 135	\\
123 210 000	& (6,7,7,38) & 114 & 123 200 300	& (7,7,7,40) & 120 & 123 230 310	 & (7,7,7,45) & 135	\\
123 230 000	& (6,7,7,40) & 120 & 120 001 302	& (7,7,7,41) & 123 & 123 012 230	 & (7,7,7,46) & 138	\\
123 231 000	& (6,7,7,40) & 120 & 120 001 310	& (7,7,7,41) & 123 & 123 210 031	 & (7,7,7,46) & 138	\\
100 200 300	& (7,4,7,28) & 84 & 120 201 300	    & (7,7,7,41) & 123 & 123 201 312	 & (7,7,7,46) & 138 	\\
100 200 010	& (7,6,3,32) & 96 & 123 200 010	    & (7,7,7,41) & 123	\\
120 200 010	& (7,6,3,36) & 108 & 120 003 310	& (7,7,7,42) & 126	\\
100 200 030	& (7,6,7,34) & 102 & 120 010 301	& (7,7,7,42) & 126	\\
120 030 300	& (7,6,7,36) & 108 & 120 010 302	& (7,7,7,42) & 126	\\
120 200 300	& (7,6,7,36) & 108 & 120 012 300	& (7,7,7,42) & 126	\\
120 010 300	& (7,6,7,38) & 114 & 120 013 300	& (7,7,7,42) & 126	\\
120 210 300	& (7,6,7,38) & 114 & 120 200 013	& (7,7,7,42) & 126	\\
120 230 300	& (7,6,7,40) & 120 & 120 203 001	& (7,7,7,42) & 126	\\
120 230 310	& (7,6,7,40) & 120 & 120 203 300	& (7,7,7,42) & 126	\\
100 010 001	& (7,7,1,28) & 84 & 123 010 300	    & (7,7,7,42) & 126	\\
\end{tabular}}
\caption{Invariants of the graph $G_1$ related to non-abelian partial algebras in $\mathcal{P}_3(\mathbb{F}_2)$.}
\label{Table_PQ3}
\end{table}

\section{Conclusion and further studies}

We have described in this paper a pair of graphs that enable us to define faithful functors between finite-dimensional algebras over finite fields and these graphs. The computation of related graph invariants plays a remarkable role in the distribution of distinct families of algebras into isotopism and isomorphism classes. Some preliminary results have been exposed in this regard, particularly on the distribution of partial quasigroup rings over finite fields. Based on the known classification of partial Latin squares into isotopism classes, further work is required to determine completely this distribution.

\end{document}